\documentclass{article}
\usepackage{graphicx} 
\usepackage[utf8]{inputenc}
\usepackage[T1]{fontenc}
\usepackage[english]{babel} 
\usepackage{enumerate}
\usepackage{import}

\usepackage{pgfplots}
\usepackage{rotating}

\usepackage{color}

\usepackage{tikz}
\usepackage{tikz-cd}

\usepackage{amssymb,amsthm,amsmath}
\usepackage{old-arrows}
\usepackage{stmaryrd}

\usepackage[sorting=none, maxnames=15]{biblatex}
\usepackage{hyperref}
\usepackage{cleveref}

\tikzset{Rightarrow/.style={double equal sign distance,>={Implies},->},
triple/.style={-,preaction={draw,Rightarrow}},
quadruple/.style={preaction={draw,Rightarrow,shorten >=0pt},shorten >=1pt,-,double,double
distance=0.2pt}}

\theoremstyle{plain}
\newtheorem{theorem}{Theorem}[subsection]

\newtheorem{proposition}[theorem]{Proposition}

\theoremstyle{definition}
\newtheorem{definition}[theorem]{Definition}
\newtheorem{notation}{Definition}[section]

\theoremstyle{remark}
\newtheorem*{remark}{Remark}

\newcommand{\R}{\mathbf{R}}

\newcommand{\K}{\mathbf{K}}

\newcommand{\End}{\mathrm{End}}
\newcommand{\EndK}{\End_{\K}}

\newcommand{\Algcat}{\mathsf{Algebra}}
\newcommand{\intervals}{\mathsf{INT}}

\newcommand{\comalg}{\mathsf{ComAlgebra}}

\newcommand{\catCN}{\mathsf{CN}}
\newcommand{\catconstCN}{\mathsf{cCN}}

\newcommand{\hombicat}{T}
\newcommand{\homcat}{S}
\newcommand{\bbox}{\scalebox{0.3}{$\blacksquare$}}

\newcommand{\radius}{0.4ex}
\newcommand{\whiteball}{{\tikz{\draw[black] (0,0) circle (\radius);}}}
\newcommand{\intw}{\intervals_\whiteball}
\newcommand{\blackball}{{\tikz{\fill[black] (0,0) circle (\radius);}}}
\newcommand{\intb}{\intervals_\blackball}
\newcommand{\blackwhiteball}{{\tikz{\draw (0,0) circle (\radius);
			\fill[black] (0,\radius) arc[start angle=90, end angle=-90, radius=\radius] -- (0,0) -- cycle;}}}
\newcommand{\intgb}{\intervals_{\blackwhiteball}}
\newcommand{\intbw}{\intervals_{\whiteball \blackball}}

\newcommand{\homog}{\mathfrak{h}}
\newcommand{\homogb}{\homog_{\black}}
\newcommand{\homogw}{\homog_{\white}}

\newcommand{\intervalsS}{\intervals^{\mathsf{circle}}}

\newcommand{\intbwS}{\intervalsS_{\whiteball \blackball}}
\newcommand{\intbwST}{\intbwS(\top)}
\newcommand{\intbwSB}{\intbwS(\bot)}

\newcommand{\col}{c}

\newcommand{\stdcirc}{\mathcal{S}^1}

\newcommand{\white}{\whiteball}
\newcommand{\black}{\blackball}
\newcommand{\blackwhite}{\blackwhiteball}
\newcommand{\blackwhiteT}{\blackwhite^\top}
\newcommand{\blackwhiteB}{\blackwhite^\bot}

\newcommand{\funb}{\Phi_\blackball}
\newcommand{\funw}{\Phi_\whiteball}
\newcommand{\funT}{\Phi^\top}
\newcommand{\funB}{\Phi^\bot}

\newcommand{\vonneumannalgebra}[1]{#1}
\newcommand{\vnA}{\vonneumannalgebra{A}}
\newcommand{\vnB}{\vonneumannalgebra{B}}
\newcommand{\vnC}{\vonneumannalgebra{C}}

\newcommand{\op}{\mathrm{op}}

\newcommand{\conformalnet}[1]{\mathcal{#1}}
\newcommand{\cfA}{\conformalnet{A}}
\newcommand{\cfB}{\conformalnet{B}}
\newcommand{\cfC}{\conformalnet{C}}

\newcommand{\defect}[1]{#1}

\newcommand{\dD}{\defect{D}}
\newcommand{\dE}{\defect{E}}
\newcommand{\dF}{\defect{F}}
\newcommand{\dG}{\defect{G}}
\newcommand{\dfus}{\circledast}
\newcommand{\sigmaw}{\sigma_\whiteball}
\newcommand{\sigmab}{\sigma_\blackball}
\newcommand{\sigmawD}{\sigmaw^\dD}
\newcommand{\sigmabD}{\sigmab^\dD}
\newcommand{\sigmawE}{\sigmaw^\dE}
\newcommand{\sigmabE}{\sigmab^\dE}
\newcommand{\sigmawF}{\sigmaw^\dF}
\newcommand{\sigmabF}{\sigmab^\dF}

\newcommand{\vacuum}{L^2}
\newcommand{\hilb}{\mathcal{H}}
\newcommand{\hilbK}{\mathcal{K}}

\newcommand{\actD}{\rho_\dD}
\newcommand{\actE}{\rho_\dE}
\newcommand{\actF}{\rho_\dF}
\newcommand{\actG}{\rho_\dG}
\newcommand{\actH}{\rho^\hilb}
\newcommand{\actK}{\rho^\hilbK}

\newcommand{\id}{\mathrm{id}}
\newcommand{\unitcat}{\mathbf{1}}

\title{Algebraic conformal nets, investigation of the locally constant case - M1 internship supervised by Domenico Fiorenza}
\author{Quentin Moreau}
\date{May-July 2025}

\begin{document}

\maketitle

\tableofcontents

\begin{abstract}
We define the tricategory of algebraic conformal nets, defects, sectors and intertwiners where algebraic refers to the absence of a topology on the relevant algebras and modules.
We aim at making these definitions satisfying from a categorical point of view, and in a few cases, e.g., for defects, we propose more general definitions than what can be found in the literature.
This allows us in particular to propose an implementation of solitons as defects.
Then we focus on the locally constant case where we provide an identification of the tricategory of locally constant algebraic conformal nets, defects, sectors and intertwiners with the tricategory of commutative algebras, algebras, bimodules and bimodule homomorphisms.
\end{abstract}

\section{Introduction}

Conformal nets play a pivotal role in the mathematical formulation of quantum field theory.
They have been widely studied and developed by Kawahigashi and Longo \cite{Longo}.
A conformal net assigns to each interval of the standard circle a von Neumann algebra acting on a Hilbert space.
Recently, Bartels, Douglas and Henriques \cite{CN-I} proposed a "coordinate-free" version of conformal nets in which intervals of the standard circle have been replaced by abstract intervals, i.e., manifolds diffeomorphic to a segment.
Henriques showed that these coordinate-free conformal nets are examples of factorization algebras \cite{CNFA}, a structure that appears in various approaches to quantum field theory.
In particular, Lurie showed that locally constant $n$-dimensional factorization algebras are in correspondence with $E_n$-algebras \cite[Theorem 5.4.5.9]{Lurie}.

Bartels, Douglas and Henriques created a whole tricategory of coordinate-free conformal nets, defects, sectors and intertwiners \cite{CN-IV}.
In \cite[Section 4.A]{CN-I}, the same authors give a correspondence between coordinate-free conformal nets and the more usual circle-based conformal nets.
However, they do not extend such a correspondence to the whole tricategory.
Moreover, concrete objects such as the standard circle still play a role in otherwise abstract definitions, leading to a conceptual competition between concreteness and abstractness.
For instance, the definition given by Bartels, Douglas and Henriques \cite[Chapter 2.A]{CN-III} of sectors relies on the standard circle.
Thus, we aim to continue this categorification of conformal nets by providing definitions that are more satisfying from a categorical point of view, while using only concrete objects.

Moreover, conformal nets and defects satisfy algebraic and topological axioms.
For instance, a defect $\dD$ should satisfy the strong additivity property: for any intervals $I$ and $J$ covering an interval $K$, $\dD(I)$ and $\dD(J)$ topologically generate $\dD(K)$.
However, when focusing on locally constant conformal nets and defects, such topological properties appear to be simpler or even trivial.
So, in order to focus on the algebraic aspects of the construction, we introduce a notion of algebraic conformal nets, defects, sectors and intertwiners which are released from topological constraints.
Then we reduce to the locally constant case.
This allows us to address the lack of explicit examples from \cite{CN-III}.
Indeed, in the locally constant case, conformal nets, defects, sectors and intertwiners have a natural algebraic interpretation: we show that in the locally constant case their tricategory is equivalent to the tricategory of commutative algebras, algebras, bimodules and bimodule homomorphisms.

\section{Notations}

\begin{notation} \label{def:algebra}
    We fix a field $\K$.
    We write $\Algcat$ the category of associative algebras over the field $\K$ whose morphisms are algebra homomorphisms.
    Any algebra will be assumed to be an object of $\Algcat$.
\end{notation}

\begin{notation}[Bicolored circle and circle intervals]
    Let $\stdcirc$ be the trigonometric circle together with the trigonometric orientation and a coloring map $\col \colon I \to \{\white, \black, \blackwhiteB, \blackwhiteT\}$ such that $\col(i) = \blackwhiteT$, $\col(-i) = \blackwhiteB$, $\col$ sends the left half of $\stdcirc$, $\{e^{i\theta}, \theta\in]\frac{\pi}{2},\frac{3\pi}{2}[\}$, on $\white$ and its right half, $\{e^{i\theta}, \theta\in]-\frac{\pi}{2},\frac{\pi}{2}[\}$, on $\black$.
    
    A circle interval is an oriented compact simply connected submanifold $I$ of $\stdcirc$ (its orientation might not be the one induced by $\stdcirc$) and such that if $I$ contains $\kappa$ then it contains a neighborhood of $\kappa$ for any $\kappa \in \{-i,i\}$.
    
    We write $\intbwS$ the category whose objects are circle intervals and morphisms are color-preserving orientation-preserving embeddings i.e., orientation-preserving embeddings $f \colon I \to J$ such that $c \circ f = c_{|I}$.
    We write $\intbwST$ (or $\intbw$) the category of upper (or bicolored) circle intervals, i.e., intervals $I$ that do not contain $-i$.
    Similarly, we write $\intbwSB$ the category of lower circle intervals.
    We write $\intw$ the category of white intervals, i.e., the intervals $I$ contained in the left half of $\stdcirc$.
    Similarly, we write $\intb$ the category of black intervals.
    We write $\intgb$ the category of genuinely bicolored intervals, i.e., upper intervals that contain $i$.

    Let us write $j \colon \stdcirc \to \stdcirc$ the orientation-reversing reflection along the horizontal axis.
    We write $\funT \colon \intbw \to \intbwST$ the identity functor.
    Moreover, we write $\funT \colon \intbw \to \intbwSB$ the functor that sends $I$ on $j(I)$, and sends $f \colon I\hookrightarrow J$ to $j \circ f \circ j$.

    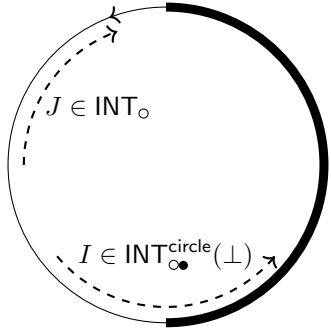
\begin{figure}[h!]
        \centering
        \begin{tikzpicture}[scale=0.7]
    
            \draw (0,0) circle(3);

            \draw[->, thick] ([shift=(110:3)]0,0) arc[start angle=110, end angle=111, radius=3];
            
            \draw[line width=1.2mm] ([shift=(-90:3)]0,0) arc[start angle=-90, end angle=90, radius=3];

            \draw[dashed, thick, ->] ([shift=(-140:2.7)]0,0) arc[start angle=-140, end angle=-40, radius=2.7];
            \node at ([shift=(-90:1.7)]0,0) {$I \in \intbwSB$};

            \draw[dashed, thick, ->] ([shift=(180:2.7)]0,0) arc[start angle=180, end angle=110, radius=2.7];
            \node at ([shift=(140:1.7)]0,0) {$J \in \intw$};

        \end{tikzpicture}
        \caption{Some circle intervals}
        \label{fig:circint}
    \end{figure}
\end{notation}

\begin{notation}[Intervals] \label{def:realintervals}
    We write $\R$ the real line oriented from the negative reals to the positive reals.
    
    An interval is an oriented compact connected submanifold of $\R$,
    that can be positively-oriented --- in which case it inherits the orientation from $\R$ --- or negatively-oriented --- in which case its orientation is opposite to the one induced by $\R$.
    We say that $J$ is a subinterval of an interval $I$ if $J \subset I$. 
    We write $\intervals$ the category whose objects are intervals and whose morphisms are orientation-preserving embeddings.
    
    We write
    \[
    \homogb \colon \begin{array}{ccc}
        \R & \longrightarrow & \{e^{i\theta}, \theta \in ]-\frac{\pi}{2},\frac{\pi}{2}[\} \\
        x & \longmapsto & \exp(i \arctan(x))
    \end{array}
    \]
    the orientation-preserving map that identifies $\R$ with the black half of $\stdcirc$.
    Similarly, we write $\homogw = - \homogb$ the orientation-preserving map that identifies $\R$ with the white half of $\stdcirc$.

    We write $\funw \colon \intervals \to \intw$ the functor that sends any interval $I$ on $\homogw(I)$ and any embedding $f$ to $\homogw \circ f \circ \homogw^{-1}$.
    Similarly, we write $\funb$ the functor induced by $\homogb$.
\end{notation}

\section{The tricategory of algebraic conformal nets} \label{sec:algCN}

\subsection{Algebraic conformal nets}

The tricategory of coordinate-free conformal nets, defects, sectors and intertwiners has been defined by Bartels, Douglas and Henriques \cite{CN-III}.
Their definitions imply topological and algebraic constraints.
In \Cref{sec:algCN}, we build algebraic conformal nets, defects, sectors and intertwiners, in order to capture only the algebraic constraints.
Therefore, our objects are released from the usual topological axioms.
Indeed, a first thing we do is to replace von Neumann algebras by abstract associative algebras, with no topological structure on them.
That is why we introduced \Cref{def:algebra}.

In the definitions of \cite{CN-I}, the notation $\vacuum$ stands for a functor that assigns to each von Neumann algebra $A$ a certain Hilbert $A$--$A$-bimodule $\vacuum A$ --- its Haagerup $\vacuum$-space defined by Haagerup \cite{haag} --- such that $\vacuum\vnA\otimes -$ is the identity of the object $A$ in the bicategory of von Neumann algebras and Hilbert bimodules. 
This is necessary because morphisms in this bicategory need to be Hilbert spaces, and the naive $A$--$A$-bimodule $A$ would not fit since a von Neumann algebra is generally not a Hilbert space.
However, in our simplified case where we consider associative algebras without topological properties, we do not have this subtlety and an algebra $A$ seen as an $A$--$A$-bimodule will be the identity functor on the object $A$ in the Morita bicategory $\mathsf{Mor}$ of algebras, bimodules and intertwiners.
The bicategory $\mathsf{Mor}$ has been introduced by Morita \cite{Morita}.
One can find a modern approach to $\mathsf{Mor}$ in the work of Brouwer \cite[Proposition 2.3.1]{Brouwer}.

Therefore, we define algebraic conformal nets as follows:

\begin{definition}[Algebraic conformal net] \label{def:algcn}
    An algebraic conformal net $\cfA$ is a functor from the category $\intervals$ to $\Algcat$ such that $\cfA(\overline{I}) = \cfA(I)^\op$ for any interval $I$ where $\overline{I}$ denotes $I$ with the reversed orientation.
    Let $I$ and $J$ be subintervals of an interval $K$.
    We write $i\colon I \hookrightarrow K$ and $j\colon J \hookrightarrow K$ the inclusion maps.
    Then the following conditions must be met:
    
    \begin{enumerate}

        \item Isotony: $\cfA(f)$ is injective for any embedding $f$.
        
        \item Locality: If $i(I)$ and $j(J)$ have disjoint interiors in $K$, then $\cfA(i)(\cfA(I))$ and $\cfA(j)(\cfA(J))$  are mutually commuting subalgebras of $\cfA(K)$.
        \label{def:algcn:it:loc}
        
        \item Strong additivity: If $K = I \cup J$ then $\cfA(K) = \cfA(i)(\cfA(I)) \vee \cfA(j)(\cfA(J))$.
    \end{enumerate}

\end{definition}

\begin{remark}
    Let us preserve the notations from \Cref{def:algcn}.
    It is often convenient to not distinguish $\cfA(I)$ and $\cfA(i)(\cfA(I))$.
    We might want then to write "$\cfA(I)$ or its image" (or only $\cfA(I)$) where "its image" implicitly refers to its image through $\cfA(i)$.
    In the next definitions, we will make such abuse of notations, but remember that our axioms never impose preservation of inclusion, i.e., $\cfA(I)$ is not necessarily a subalgebra of $\cfA(K)$.
\end{remark}

\begin{remark}
    Let $\cfA$ be a conformal net and $f\colon I \hookrightarrow J$ be an orientation-reversing embedding.
    Then $f$ induces an orientation-preserving embedding $\overline{f} \colon I \hookrightarrow \overline{J}$ where $\overline{J}$ denotes $J$ with the reversed orientation.
    Hence, $\cfA(\overline{f}) \colon \cfA(I) \to \cfA(J)^\op$ is well-defined and induces a linear antihomomorphism $\cfA(f) \colon \cfA(I) \to \cfA(J)$.
\end{remark}

\subsection{Algebraic defects}

We now give an algebraic definition of defects.
We aim to provide a more categorical definition than the one given by \cite[Definition 1.7]{CN-III} thanks to the functors defined in \Cref{def:realintervals} and the use of natural isomorphisms:

\begin{definition}[Algebraic defects] \label{def:algdef}
    Let $\cfA$ and $\cfB$ be algebraic conformal nets.
    An algebraic defect between $\cfA$ and $\cfB$ (or an algebraic $\cfA$--$\cfB$ defect) is a functor $\dD$ from $\intbw$ to $\Algcat$ together with natural isomorphisms $\sigmaw \colon \cfA \Leftrightarrow \dD \circ \funw$ and $\sigmab \colon \cfB \Leftrightarrow \dD \circ \funb$:
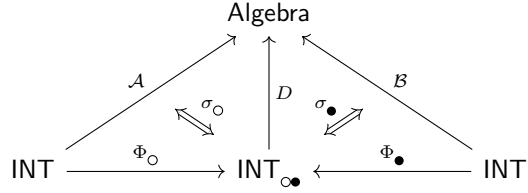
    \[
    \begin{tikzcd}
    & & \Algcat & & \\
    & {} \arrow[rd, "\sigmaw", Leftrightarrow, shorten >=10pt, shorten <=10pt] & & {} \arrow[ld, "\sigmab"', Leftrightarrow, shorten >=10pt, shorten <=10pt] & \\
    \intervals \arrow[rr, "\funw"] \arrow[rruu, "\cfA"] & & \intbw \arrow[uu, "\dD"'] & & \intervals \arrow[ll, "\funb"'] \arrow[lluu, "\cfB"']
    \end{tikzcd}
    \]
    The functor $\dD$ is such that $\dD(\overline{I}) = \dD(I)^\op$ for any interval $I$ where $\overline{I}$ denotes $I$ with the reversed orientation.
    Then for any bicolored subintervals $I$, $J$ of an interval $K$, the following conditions must be met:
    \begin{enumerate}
        \item Isotony: If $f \colon I \hookrightarrow J$ is an embedding of genuinely bicolored intervals, then $\dD(f)$ is injective.
        
        \item Locality: \label{def:algdef:it:loc}
        If $I$ and $J$ have disjoint interiors, then the images of $\dD(I)$ and $\dD(J)$ are mutually commuting subalgebras of $\dD(K)$.

        \item Strong additivity: \label{def:algdef:it:sa}
        If $K = I \cup J$ then $\dD(K) = \dD(I) \vee \dD(J)$.
    \end{enumerate}

\end{definition}

For a defect $(\dD,\sigmaw,\sigmab)$, we will write simply $\dD$ when $\sigmaw$ and $\sigmab$ are clear from the context.

\begin{remark}
    In \cite{CN-III}, it is suggested that a soliton (an endomorphism localized in a half-line) could be implemented as a defect.
    Indeed, Henriques established deep connections between solitons and conformal nets in \cite{bicommutantcat} and \cite{ChernSimons}.
    However, it is not clear what a coordinate-free soliton would be.
    Here, with the categorical definition proposed in \Cref{def:algdef}, a soliton can be thought of as a particular case of an $\cfA$--$\cfA$ defect $(\dD,\sigmaw,\sigmab)$ in which $\sigmaw$ or $\sigmab$ is the identity.
    If both $\sigmaw$ and $\sigmab$ are the identity, then we recover the definition of \cite{CN-III}.
    Therefore, it is natural to relax the condition of commutativity of the diagram from \Cref{def:algdef} and ask only commutativity up to natural isomorphisms.
    Moreover, this categorical definition might be used to avoid case-by-case definitions such as the one given by \cite{Janssens}.
\end{remark}

The fusion $\vnA \dfus_\vnC \vnB$ of two von Neumann algebras $\vnA$ and $\vnB$ over a third one $\vnC$ is defined in terms of faithful Hilbert modules for $\vnA$ and $\vnB$ in \cite[Definition 1.20]{CN-I}.
In our algebraic setting, we are allowed to consider the action of an algebra $\vnA$ on itself through multiplication.
Thus, this simplifies to the following:

\begin{definition}[Fusion of associative algebras]\label{def:fusvn}
    Let $\vnA$, $\vnB$ and $\vnC$ be associative algebras such that we are given homomorphisms of algebras $\vnC^\op \to \vnA$ and $\vnC \to \vnB$.
    Then we define the fusion $\vnA \dfus_\vnC \vnB$ of $\vnA$ and $\vnB$ over $\vnC$ as $(\vnA \cap {\vnC^{\op}}') \vee (\vnC' \cap \vnB)$ seen as a subalgebra of $\EndK(\vnA \otimes_\vnC \vnB)$ and where the commutants of $\vnC^{\op}$ and $\vnC$ are taken respectively in $\vnA$ and $\vnB$.
\end{definition}

\begin{definition}[Fusion of defects] \label{def:fusdef}
    Let $\cfA$, $\cfB$ and $\cfC$ be algebraic conformal nets.
    Let $(\dD,\sigmawD,\sigmabD)$ be an $\cfA$--$\cfB$ defect and $(\dE,\sigmawE,\sigmabE)$ be a $\cfB$--$\cfC$ defect.
    We define their fusion $\dD \dfus_\cfB \dE$.

    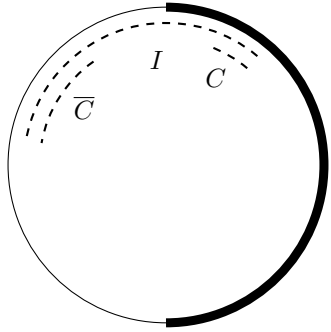
\begin{figure}[h!]
        \centering
        \begin{tikzpicture}[scale=0.7]
    
            \draw (0,0) circle(3);
            
            \draw[line width=1.2mm] ([shift=(-90:3)]0,0) arc[start angle=-90, end angle=90, radius=3];

            \draw[dashed, thick] ([shift=(50:2.7)]0,0) arc[start angle=50, end angle=170, radius=2.7];
            \node at ([shift=(95:2)]0,0) {$I$};

            \draw[dashed, thick] ([shift=(50:2.4)]0,0) arc[start angle=50, end angle=70, radius=2.4];
            \node at ([shift=(60:1.9)]0,0) {$C$};

            \draw[dashed, thick] ([shift=(125:2.4)]0,0) arc[start angle=125, end angle=170, radius=2.4];
            \node at ([shift=(145:1.9)]0,0) {$\overline{C}$};

        \end{tikzpicture}
        \caption{Fusion of defects}
        \label{fig:fusdef}
    \end{figure}

    Let $I$ be a bicolored interval.
    If $I$ is white then $\dD \dfus_\cfB \dE (I) = \dD(I)$.
    If $I$ is black then $\dD \dfus_\cfB \dE (I) = \dE(I)$.

    Let us suppose $I$ genuinely bicolored.
    Then, write $\partial I = \{e^{i(\frac{\pi}{2}-\theta_1)},e^{i(\frac{\pi}{2}+\theta_2)}\}$ such that $\theta_1$ and $\theta_2$ belong to $]0,\pi[$.
    We write $C = \{e^{i(\frac{\pi}{2}-\theta)}, \theta \in [\frac{\theta_1}{2},\theta_1]\}$ and $\overline{C} = \{e^{i(\frac{\pi}{2}+\theta)}, \theta \in [\frac{\theta_2}{2},\theta_2]\}$ (see \Cref{fig:fusdef}).
    We write
    \[
    f \colon
    \begin{array}{ccc}
        C & \longrightarrow & \overline{C} \\
        \exp(i(\frac{\pi}{2}-\theta)) & \longmapsto & \exp(i(\frac{\pi}{2}+\frac{\theta_2}{\theta_1}\theta))
    \end{array}
    \]
    the orientation-reversing map that sends $C$ onto $\overline{C}$.
    Then we write $J$ the preimage of $C$ under $\homogb$.
    Then, $\sigmaw(J)$ induces an algebra homomorphism from $\cfB(J)$ to $\dE(I)$ and $\sigmaw(\homogw^{-1} \circ f \circ \homogb (J))$ induces an algebra antihomomorphism from $\cfB(J)$ to $\dD(I)$.

    Therefore, we can define $\dD \dfus_\cfB \dE (I)$ as $\dD(I) \dfus_{\cfB(J)} \dE(I)$.

\end{definition}

In \cite{CN-III}, it is conjectured that the fusion of two defects gives a defect.
This conjecture is proved when $\cfB$ has finite index.
However, this hypothesis is used only to prove that $\dD \dfus_\cfB \dE$ satisfies the "vacuum sector" property that is absent from our algebraic setting.
Therefore, we can state the following:

\begin{proposition}
    The fusion of two algebraic defects is an algebraic defect.
\end{proposition}

\subsection{Algebraic sectors and intertwiners}

In this subsection, we define algebraic sectors and intertwiners.

Let $\mathsf{Core(Vect)}$ be the core of the category of vector spaces over $\K$, i.e., the category whose objects are vector spaces over $\K$ and whose morphisms are linear isomorphisms. We have a distinguished functor
\[
\EndK\colon \mathsf{Core(Vect)} \to \Algcat
\]
mapping a vector space $V$ to its algebra of endomorphisms $\EndK(V)$ and an isomorphism $f\colon V\to W$ to the following algebra isomorphism:
\[
\begin{array}{ccc}
\EndK(V) & \to & \EndK(W)\\
\phi & \mapsto & f\circ \phi\circ f^{-1}
\end{array}.
\]

\begin{definition}[Algebraic sectors] \label{def:sector}
    Let $(\dD,\sigmawD,\sigmabD)$ and $(\dE,\sigmawE,\sigmabE)$ be $\cfA$--$\cfB$ defects.
    We define an algebraic sector between $\dD$ and $\dE$ (or a $\dD$--$\dE$ sector) as a constant functor $\hilb \colon \intbwS \to \mathsf{Core(Vect)}$ together with natural transformations $\actD \colon \dD \Rightarrow \EndK \circ \hilb \circ \funT$ and $\actE \colon \dE \Rightarrow \EndK \circ \hilb \circ \funB$ as shown in the following diagram:
    \[
    \begin{tikzcd}[column sep=0.5cm, row sep=0.3cm]
    \intbw \arrow[ddddd, "\funT"'] \arrow[rrrrrrddddd, "\dD"] &    &                                     &  &  &   &      \\
    &    &                                     &  &  &      &   \\
    &    &                                     &  &  &     &    \\
    &    & {} \arrow[lldd, "\actD"', Rightarrow, shorten >=10pt, shorten <=10pt] &  &  &         \\
    & {} &                                     &  &  &    &     \\
    \intbwS \arrow[rrr, "\hilb"]                                    &    &                                     &\mathsf{Core(Vect)}\arrow[rrr, "\EndK"]  &  & &\Algcat \\
    & {} &                                     &  &  &   &      \\
    &    & {} \arrow[lluu, "\actE", Rightarrow, shorten >=10pt, shorten <=10pt]  &  &  &         \\
    &    &                                     &  &  &   &      \\
    &    &                                     &  &  &    &     \\
    \intbw \arrow[uuuuu, "\funB"] \arrow[rrrrrruuuuu, "\dE"'] &    &                                     &  &  &       & 
    \end{tikzcd}
    \]
    For any inclusion of circle intervals $i \colon J \hookrightarrow I$, such that $I \in \intbwST$, we require that a natural \textit{compatibility} isomorphism $\actD(I) \circ \dD(i) \Leftrightarrow \hilb(i) \circ \actD(J)$ exists as shown in the following diagram:
    \[
    \begin{tikzcd}
    & \dD(I) \arrow[rd, "\actD(I)"] \arrow[dd, Leftrightarrow, shorten >=15pt, shorten <=15pt] &          \\
    \dD(J) \arrow[rd, "\actD(J)"'] \arrow[ru, "\dD(i)"] &                                                      & \hilb(I) \\
    & \hilb(J) \arrow[ru, "\hilb(i)"']                     &         
    \end{tikzcd}
    \]
    The analogous condition holds for $\dE$ and $\intbwSB$.

    Moreover, $(\hilb,\actD,\actE)$ must satisfy the following locality property for any upper and lower, respectively, circle subintervals $I$ and $J$ of a tricolored interval $K$: if $I$ and $J$ have disjoint interiors, then the images of $D(I)$ and $E(J)$ are mutually commuting subalgebras of $\EndK(\hilb(K))$.
\end{definition}

For a sector $(\hilb,\actD,\actE)$, we will write $(\hilb,\actH)$ or even just $\hilb$ when the natural isomorphisms are clear from the context.

\begin{remark}
    Let us consider a sector $(\hilb,\actD,\actE)$ and a bicolored interval $I$.
    Notice then that the morphism $\actD\colon D(I)\to \EndK(\hilb(I))$ endows $\hilb(I)$ with a structure of left $D(I)$-module. Similarly, due to the orientation reversal in the definition of $\funB$, the morphism $\actE$ makes $\hilb(I)$ a right $E(I)$-module.
    This makes $\hilb(I)$ a $D(I)$--$E(I)$-bimodule.
\end{remark}

Thanks to the previous remark, we can now define vertical and horizontal fusions:

\begin{definition}[Vertical fusion of sectors] \label{def:vfussec}
    Let $\cfA$ and $\cfB$ be conformal nets, and $(\dD,\sigmawD,\sigmabD)$, $(\dE,\sigmawE,\sigmabE)$ and $(\dF,\sigmawF,\sigmabF)$ be $\cfA$--$\cfB$ defects.
    Let $(\hilb,\actH)$ and $(\hilbK,\actK)$ be $\dD$--$\dE$ and $\dE$--$\dF$ sectors respectively.
    We define the vertical fusion of $\hilb$ and $\hilbK$ by the following diagram:
 \[
    \begin{tikzcd}[column sep=0.5cm, row sep=0.3cm]
    \intbw \arrow[ddddd, "\funT"'] \arrow[rrrrrrddddd, "\dD"] &    &                                     &  &  &   &      \\
    &    &                                     &  &  &      &   \\
    &    &                                     &  &  &     &    \\
    &    & {} \arrow[lldd, "\actD\otimes \mathrm{id}"', Rightarrow, shorten >=10pt, shorten <=10pt] &  &  &         \\
    & {} &                                     &  &  &    &     \\
    \intbwS \arrow[rrr, "\hilb\otimes_{\dE(I)}\hilbK"]                                    &    &                                     &\mathsf{Core(Vect)}\arrow[rrr, "\EndK"]  &  & &\Algcat \\
    & {} &                                     &  &  &   &      \\
    &    & {} \arrow[lluu, "\mathrm{id}\otimes \actF", Rightarrow, shorten >=10pt, shorten <=10pt]  &  &  &         \\
    &    &                                     &  &  &   &      \\
    &    &                                     &  &  &    &     \\
    \intbw \arrow[uuuuu, "\funB"] \arrow[rrrrrruuuuu, "\dF"'] &    &                                     &  &  &       & 
    \end{tikzcd}
    \]
    where $I$ is the genuinely bicolored interval $\{e^{i\theta}, \theta\in[0,\pi]\}$.

\end{definition}

\begin{definition}[Horizontal fusion of sectors] \label{def:hfussec}
    Let $\cfA$, $\cfB$ and $\cfC$ be conformal nets, $\dD$ and $\dE$ be $\cfA$--$\cfB$ defects and $\dF$ and $\dG$ be $\cfB$--$\cfC$ defects.
    Let $\hilb$ be a $\dD$--$\dE$ sector and $\hilbK$ be a $\dF$--$\dG$ sector.
    We define the horizontal fusion of $\hilb$ and $\hilbK$ by the following diagram:
 
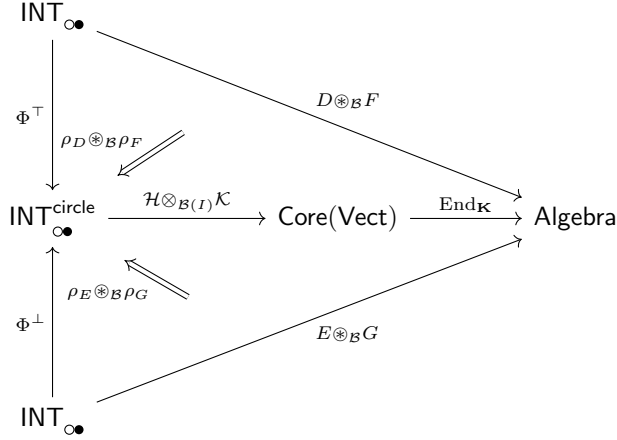
\[
    \begin{tikzcd}[column sep=0.5cm, row sep=0.3cm]
    \intbw \arrow[ddddd, "\funT"'] \arrow[rrrrrrddddd, "\dD \dfus_\cfB \dF"] &    &                                     &  &  &   &      \\
    &    &                                     &  &  &      &   \\
    &    &                                     &  &  &     &    \\
    &    & {} \arrow[lldd, "\actD \dfus_\cfB \actF"', Rightarrow, shorten >=10pt, shorten <=10pt] &  &  &         \\
    & {} &                                     &  &  &    &     \\
    \intbwS \arrow[rrr, "\hilb\otimes_{\cfB(I)} \hilbK"]                                    &    &                                     &\mathsf{Core(Vect)}\arrow[rrr, "\EndK"]  &  & &\Algcat \\
    & {} &                                     &  &  &   &      \\
    &    & {} \arrow[lluu, "\actE \dfus_\cfB \actG", Rightarrow, shorten >=10pt, shorten <=10pt]  &  &  &         \\
    &    &                                     &  &  &   &      \\
    &    &                                     &  &  &    &     \\
    \intbw \arrow[uuuuu, "\funB"] \arrow[rrrrrruuuuu, "\dE \dfus_\cfB \dG"'] &    &                                     &  &  &       & 
    \end{tikzcd}
    \]
    where $\actD \dfus_\cfB \actF$ and $\actE \dfus_\cfB \actG$ are the natural isomorphisms induced by $\actD$, $\actF$, $\actE$ and $\actG$ and $I$ is the white interval $\{e^{i\theta}, \theta\in[\frac{3\pi}{4},\frac{5\pi}{4}]\}$.

\end{definition}

\begin{definition}[Algebraic intertwiners]
    Let $\cfA$ and $\cfB$ be conformal nets.
    Let $\dD$ and $\dE$ be $\cfA$--$\cfB$ defects.
    Let $(\hilb,\actH)$ and $(\hilbK,\actK)$ be $\dD$--$\dE$ defects.
    An intertwiner between $\hilb$ and $\hilbK$ is a linear map $\phi\colon \hilb \to \hilbK$ such that $\phi \circ \actH = \actK \circ \phi$.
\end{definition}


\begin{definition}[Transversal fusion of intertwiners]\label{def:tfusint}
    Let $\phi\colon \hilb \to \hilb'$ and $\psi\colon \hilb' \to \hilb''$ be intertwiners. We define their transversal fusion as the intertwiner $\psi \circ \phi$.
\end{definition}

\begin{definition}[Vertical fusion of intertwiners]\label{def:vfusint}
    Let $\cfA$ and $\cfB$ be conformal nets, $\dD$, $\dE$, and $\dF$ be $\cfA$--$\cfB$ defects, $\hilb$ and $\hilbK$ be $\dD$--$\dE$ sectors, $\hilb'$ and $\hilbK'$ be $\dE$--$\dF$ sectors.
    For any intertwiners $\phi\colon \hilb \to \hilbK$ and $\psi\colon \hilb' \to \hilbK'$, we define their vertical fusion $\phi \otimes_\dE \psi$ as the intertwiner from $\hilb \otimes_\cfB \hilbK$ to $\hilb' \otimes_\cfB \hilbK'$ induced by $\phi$ and $\psi$.
\end{definition}

\begin{definition}[Horizontal fusion of intertwiners]\label{def:hfusint}
    Let $\cfA$, $\cfB$ and $\cfC$ be conformal nets, $\dD$ and $\dE$ be $\cfA$--$\cfB$ defects, $\dF$ and $\dG$ be $\cfB$--$\cfC$ defects, $\hilb$ and $\hilb'$ be $\dD$--$\dE$ sectors and $\hilbK$ and $\hilbK'$ be $\dF$--$\dG$ sectors.
    For any intertwiners $\phi\colon \hilb \to \hilb'$ and $\psi\colon \hilbK \to \hilbK'$, we define their horizontal fusion $\phi \otimes_\cfB \psi$ as the intertwiner from $\hilb \otimes_\dE \hilbK$ to $\hilb' \otimes_\dE \hilbK'$ induced by $\phi$ and $\psi$.
\end{definition}

So far we have recalled the main definitions from \cite{CN-III}, adapting them to an algebraic and circle-based context. The arguments in \cite{CN-III} and \cite{CN-IV} verbatim adapt (actually simplify) to this context, so that the main result of \cite{CN-IV} provides the following:
\begin{theorem}
     Algebraic conformal nets, algebraic defects, algebraic sectors and algebraic intertwiners form a tricategory.
\end{theorem}
We will denote this tricategory by the symbol $\catCN$.

\section{The tricategory $\catconstCN$ of locally constant algebraic conformal nets} \label{sec:loccons}

Topological constraints on conformal nets appear to be trivial or at least simpler in the locally constant case.
Thus, the algebraic tricategory $\catCN$ is a natural setting where to  study the locally constant case.
In \Cref{sec:loccons}, we aim at giving a completely algebraic description of the  subcategory of $\catCN$ consisting of locally constant algebraic conformal nets.

\subsection{Reducing $\catCN$ to the locally constant case} \label{ssec:loccons}

First, we investigate what objects and morphisms from $\catCN$ reduce to under the locally constancy condition.

For conformal nets, we get:

\begin{definition}[Locally constant conformal nets]
    A conformal net $\cfA$ is locally constant if $\cfA$ maps any positively-oriented interval to the same algebra $\vnA$ and all embeddings of positively-oriented intervals to the identity of $\vnA$.
\end{definition}
Since 
all intervals are diffeomorphic to each other, we see that a locally constant algebraic conformal net is in particular the datum of an algebra $\vnA$ together with an involutive antiisomorphism $\overline{(\,\,)}_\vnA\colon \vnA\to \vnA^\op$: for every interval $I$ one has $\cfA(I)=\vnA$, for every
embedding $i\colon I\to J$ of intervals with same orientation (both positively or both negatively-oriented) we have $\cfA(i)=\id_{\vnA}$ and for every
embedding $\overline{i}\colon I\to J$ of intervals with opposite orientations we have $\cfA(\overline{i})=\overline{(\,\,)}_\vnA$. For this reason we will simply write $\vnA$ instead of $\cfA$ for a locally constant algebraic conformal net.

A first thing we can notice is the following:

\begin{proposition} \label{prop:constCNcomm}
    The unique associative algebra $A$ in a locally constant conformal net is commutative.
\end{proposition}

\begin{proof}
    This is a direct consequence of the locality property.
\end{proof}

Since $\vnA$ is commutative, we have $\vnA^\op = \vnA$, and so $\overline{(\,\,)}_\vnA\colon \vnA\to \vnA$ is an involutive algebra automorphism. In particular, $\overline{(\,\,)}_\vnA= \id_\vnA$ is a valid choice. In this case one has that $\cfA$ maps every morphism in $\intervals$ to the identity of $\vnA$. When this happens we say that $\cfA$ is a locally constant \emph{unoriented} algebraic conformal net. In what follows we will only deal with the unoriented case. The definitions and constructions for the oriented case are verbatim adapted from the unoriented case.

Because of \Cref{prop:constCNcomm}, we will talk about commutative algebras instead of associative algebras when dealing with a locally constant conformal net.

Then defects reduce to the following in the locally constant case:

\begin{definition}[Locally constant defects]\label{def:lcdefect}
    Let $\vnA$ and $\vnB$ be locally constant conformal nets.
    A defect $(\dD,\sigmaw,\sigmab)$ between $\vnA$ and $\vnB$ is said to be locally constant if  $D$ maps every orientation-preserving morphism in any of the subcategories $\intw$, $\intb$, and $\intgb$ of $\intbw$ to an identity morphism, is constant on the set of morphisms from objects in $\intw$ to  objects in $\intgb$, is constant on the set of morphism from objects in $\intb$ to  ojects in $\intgb$, and
    $\sigmaw$ and $\sigmab$ are identities.
\end{definition}

\begin{remark}
   In what follows a defect between locally constant conformal nets will always be assumed to be locally constant.
\end{remark}

As constant conformal nets can be seen as commutative algebras, there is a natural algebraic description of locally constant defects.  This is the content of the following proposition.

\begin{proposition}
    For any constant conformal nets $\vnA$ and $\vnB$, defects between $\vnA$ and $\vnB$ are pairs $(D,\varphi)$ where $D$ is an algebra and $\varphi\colon \vnA \otimes \vnB\to Z(D)$ is an algebra homomorphism from $\vnA \otimes \vnB$ to the center of $D$.
\end{proposition}

\begin{proof}
    Let $\vnA$ and $\vnB$ be constant conformal nets.
    It is then clear that any $(D,\varphi)$  as in the statement defines a locally constant defect. For the inclusion of a white resp. black, interval into a genuinely bicolored interval, the morphism $A\to D$, resp. $B\to D$, is given by the restriction of $\varphi$ to $A\otimes 1_B$, resp. to $1_A\otimes B$.
    
    \begin{figure}[h!]
        \centering
        \begin{tikzpicture}[scale=0.7]
    
            \draw (0,0) circle(3);
            
            \draw[line width=1.2mm] ([shift=(-90:3)]0,0) arc[start angle=-90, end angle=90, radius=3];

            \draw[dashed, thick] ([shift=(140:2.5)]0,0) arc[start angle=140, end angle=40, radius=2.5];
            \node at ([shift=(90:2)]0,0) {$I$};

            \draw[dashed, thick] ([shift=(-110:2.5)]0,0) arc[start angle=-110, end angle=-180, radius=2.5];
            \node at ([shift=(-150:2)]0,0) {$J$};
            
            \draw[dashed, thick] ([shift=(-100:3.5)]0,0) arc[start angle=-100, end angle=-330, radius=3.5];
            \node at ([shift=(-190:4)]0,0) {$K$};

        \end{tikzpicture}
        \caption{Locally constant defect}
        \label{fig:loccd}
    \end{figure}
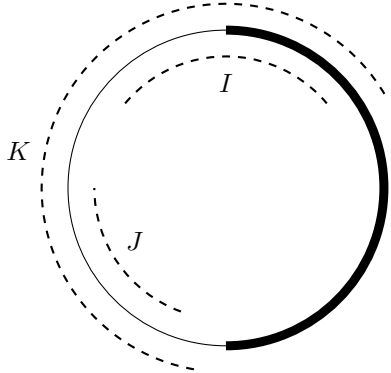

    Vice versa, let $\dD$ be a locally constant defect between $\vnA$ and $\vnB$.
    Let $I$ be a genuinely bicolored interval and $J$ be a white interval disjoint from $I$.
    Let $K$ be a genuinely bicolored interval that contains $I$ and $J$ (see \Cref{fig:loccd}).
    Then by the locality property, $\dD(J)=A$ and $\dD(I)=D$ are mutually commuting subalgebras of $\dD(K)=D$.
    In other words, the morphism $\vnA\to D$ induced by the embedding $J\hookrightarrow K$ maps $\vnA$ into the center of $\dD$.
    Similarly, $\vnB$ is mapped into the center of $\dD(I)$. From the universal property of the tensor product of commutative rings we then have that this datum is equivalent to that of an algebra homomorphism $\varphi\colon \vnA \otimes \vnB\to Z(D)$.

\end{proof}

Furthermore, when fusing locally constant defects, we get the following:

\begin{proposition} \label{prop:fusloccdef}
    Let $\vnA$, $\vnB$ and $\vnC$ be constant conformal nets and $\dD$ and $\dE$ be $\vnA$--$\vnB$ and $\vnB$--$\vnC$ defects.
    Then $\dD \dfus_\vnB \dE = \dD \otimes_\vnB \dE$, where $\dD \otimes_\vnB \dE$ acts on itself by left multiplication.
\end{proposition}

\begin{proof}
    Since $\vnB$ is both contained in the center of $\dD$ and in the center of $\dE$, we have $\dD \dfus_\vnB \dE=\dD \vee \dE$, so we have to prove that $\dD \otimes_\vnB \dE = \dD \vee \dE$.
    In order to do so, we exhibit a natural morphism $\dD \otimes_\vnB \dE \to \EndK(\dD \otimes_\vnB \dE)$, and show that it is an isomorphism onto $\dD \vee \dE$.
    
    Consider the map
    \[
    \rho\colon \dD \times \dE\to  \EndK(\dD \otimes_\vnB \dE)
    \]
    defined by
    \[
    \forall (d,\delta,e,\epsilon) \in \dD^2 \times \dE^2
    \rho_{(d,e)}(\delta\otimes_\vnB \epsilon)=d\delta\otimes_\vnB  e \epsilon.
    \]
    
    The map $\rho$ is manifestly bilinear and for any $(d,\delta,b,e,\epsilon) \in \dD^2 \times \vnB \times \dE^2$, we have:
    \[
    \rho_{(db,e)}(\delta \otimes_\vnB \epsilon)
    = db\delta \otimes_\vnB e \epsilon
    = d\delta b\otimes_\vnB e \epsilon
    = d\delta \otimes_\vnB b e \epsilon
    = \rho_{(d,be)}(\delta \otimes_\vnB \epsilon),
    \]
    where we use the fact that $\vnB$ is a subalgebra of the center of $\dD$ and of the center of $\dE$.
    Moreover, for any $(d,\delta,e,\epsilon) \in \dD^2 \times \dE^2$, we have $\rho_{(d\delta, e\epsilon)} = \rho_{(d,e)} \circ \rho_{(\delta,\epsilon)}$.
    Therefore, $\rho$ defines an algebra homomorphism $\rho\colon \dD\otimes_\vnB \dE\to \EndK(D \otimes_\vnB E)$.
    
    The image of $\rho$ is manifestly generated by left multiplications by elements of $\dD$ and of $\dE$.
    Hence, $\rho$ takes values in $\dD \vee \dE$ and, since the image of $\rho$ contains all of these multiplications, and the latter generate $\dD \vee \dE$, we have that $\rho\colon \dD\otimes_\vnB \dE\to \dD\vee \dE$ is surjective.
    It remains to show that $\rho$ is injective.
    This is readily seen by noticing that for any $\xi\in \dD\otimes_\vnB\dE$ one has $\rho_\xi(1\otimes_\vnB 1)=\xi$.
    
\end{proof}

For any constant defects $\dD$ and $\dE$, a $\dD$--$\dE$ sector is a $\dD$--$\dE$ bimodule.
In particular, we impose that natural compatibility isomorphisms are identities.
When focusing on locally constant sectors, fusions of sectors are tensor products over algebras since our definition reduces to the one of \cite{CN-III}.



In the locally constant case, an intertwiner is a bimodule homomorphism.

Now, to summarize, we have built the following subcategory of $\catCN$:

\begin{definition}
    The tricategory $\catconstCN$ is the one whose objects are constant conformal nets, $1$-morphisms are locally constant defects, $2$-morphisms are locally constant sectors and $3$-morphisms are intertwiners.
    Compositions are defined by \Cref{prop:fusloccdef},
    tensor products over algebras,
    \Cref{def:tfusint}, \Cref{def:hfusint} and \Cref{def:vfusint}.
    Moreover, the identity of a constant conformal net $\vnA$ is itself seen as an $\vnA$--$\vnA$ defect, the identity of a defect $\dD$ is itself seen as a $\dD$--$\dD$ sector and the identity of a sector $\hilb$ is the identity map $\mathrm{id}_\hilb$.
\end{definition}

\subsection{The tricategory of commutative algebras} \label{ssec:comalg}

In this subsection, we draw inspiration from Carqueville and Müller \cite[Theorem 5.21]{Orbifold} to define a tricategory $\comalg$.
For a complete definition of a tricategory, see Gurski’s thesis \cite[Definition 3.1.2 and Appendix B.1]{Tricategory} and for a complete definition of a bicategory, see Bénabou \cite{Bicategory}.
The tricategory $\comalg$ can be seen as a decategorification of the tetracategory of braided tensor categories introduced by Freed, Hopkinks, Lurie and Teleman \cite{topQFTcompLieg}. Indeed, the description of this tetracategory in the form recalled by Fiorenza and Valentino in \cite[Section 8.1]{FiorenzaValentino} is another source of inspiration for the definition of the tricategory $\comalg$.

\subsubsection{The bicategory of morphisms} \label{subsec:bic}

To define a tricategory, we must first define a bicategory which needs itself to define a category.
This subsection aims at providing the bicategory of morphisms which is an intermediary step in the creation of a tricategory.

\begin{definition} \label{def:homcat}
    Let $\vnA$ and $\vnB$ be commutative algebras.
    Let $\dD$ and $\dE$ be associative algebras together with homomorphisms $\alpha\colon \vnA \otimes \vnB \to Z(\dD)$ and $\beta\colon \vnA \otimes \vnB \to Z(\dE)$.
    We define a category $\homcat(\dD,\alpha,\dE,\beta)$ (or $\homcat(\dD,\dE)$ when $\alpha$ and $\beta$ are clear from the context) as follows:

    \begin{enumerate}
        \item Objects: An object is a $\dD$--$\dE$ bimodule $\hilb$ such that the actions of $\vnA \otimes \vnB$ induced by $\alpha$ and $\beta$ are equal.
        \item Morphisms: A morphism between $\hilb$ and $\hilbK$ is a bimodule homomorphism.
        \item Composition: The composition of morphisms $\phi \colon \hilb \to \hilb'$ and $\psi \colon \hilb' \to \hilb''$ is $\psi \circ \phi \colon \hilb \to \hilb''$.
        \item Identity: The identity of $\hilb$ is $\id_\hilb$, the identity map.
   \end{enumerate}
\end{definition}

\begin{definition} \label{def:comphomcat}
    Let $\vnA$ and $\vnB$ be commutative algebras and $(\dD,\alpha)$, $(\dE,\beta)$ and $(\dF,\gamma)$ be couples of associative algebras and homomorphisms from $\vnA \otimes \vnB$ to the center of $\dD$, $\dE$ and $\dF$ respectively.
    We define the composition functor $\otimes_\dE$ as follows:
    \[
    \otimes_\vnB \colon
    \begin{array}{ccc}
        \homcat(\dE,\dF) \times \homcat(\dD,\dE) & \longrightarrow & \homcat(\dD,\dF) \\
        (\hilbK,\hilb) & \longmapsto & \hilb \otimes_\dE \hilbK \\
        (\psi\colon \hilbK \to \hilbK',\phi\colon \hilb\to\hilb') & \longmapsto & \phi \otimes_\dE \psi
    \end{array},
    \]
    where the tensor products over $\dE$ are well-defined since we consider left and right $\dE$-modules and $\dE$-module homomorphisms.
\end{definition}

\begin{definition} \label{def:hombicat}
    Let $\vnA$ and $\vnB$ be commutative algebras.
    We define a bicategory $\hombicat(\vnA,\vnB)$ of morphisms between them as follows:

    \begin{enumerate}
        \item Objects: An object is an associative algebra $\dD$ together with a homomorphism from $\vnA \otimes \vnB$ to $Z(\dD)$.
        
        \item Hom-category: The hom-category associated to two objects $(\dD,\alpha)$ and $(\dE,\beta)$ is $\homcat(\dD,\alpha,\dE,\beta)$ defined in \Cref{def:homcat}.

        \item Composition functor: \label{def:hombicat:comp}
        For objects $(\dD,\alpha)$, $(\dE,\beta)$ and $(\dF,\gamma)$, the composition functor $\otimes_\dE\colon \homcat(\dE,\dF) \times \homcat(\dD,\dE) \to \homcat(\dD,\dF)$ is defined in \Cref{def:comphomcat}.

        \item Identity arrow: \label{def:hombicat:idarr}
        For each object $(\dD,\alpha)$, its identity arrow is $\dD$ seen as a $\dD$--$\dD$ bimodule.

        \item Associativity morphism: \label{def:hombicat:pent}
        For each quadruple of objects $(\dD,\alpha)$, $(\dE,\beta)$, $(\dF,\gamma)$ and $(\dG,\delta)$, the associativity natural isomorphism $a$ associates to any $(\hilb,\hilb',\hilb'') \in \homcat(\dF,\dG) \times \homcat(\dE,\dF) \times \homcat(\dD,\dE)$ the following isomorphism:
        \[
        \begin{array}{ccc}
            (\hilb \otimes_\dE \hilb') \otimes_\dF \hilb'' & \longrightarrow & \hilb \otimes_\dE (\hilb' \otimes_\dF \hilb'') \\
            (h \otimes_\dE h') \otimes_\dF h'' & \longmapsto & h \otimes_\dE (h' \otimes_\dF h'')
        \end{array}.
        \]

        \item Left and right identities: \label{def:hombicat:lrid}
        For each pair of objects $(\dD,\alpha)$ and $(\dE,\beta)$, the left identity $l$ is the natural isomorphism that associates to any $(\hilb,1)$ from $\homcat(\dD,\dE) \times \unitcat$ the following isomorphism:
        \[
        \begin{array}{ccc}
            \hilb & \longrightarrow & \dD \otimes_\dD \hilb \\
            h & \longmapsto & 1 \otimes_\dD h
        \end{array}
        \]
        
        The right identity $r$ is defined similarly. 
    
    \end{enumerate}
\end{definition}

\begin{proof}[Proof of the well-definition]

The associativity and identity coherence axioms from \cite{Bicategory} are satisfied by \Cref{def:hombicat} thanks to the commutativity of the following diagrams:
\[
\begin{tikzcd}[row sep=1.5cm, column sep=0.7cm]
((\hilb \otimes_{\dD'} \hilb') \otimes_{\dD''} \hilb'') \otimes_{\dD'''} \hilb''' \arrow[rr, "a \otimes_{\dD'''} \id"] \arrow[d, "a"'] &                                                                                  & (\hilb \otimes_{\dD'} (\hilb' \otimes_{\dD''} \hilb'')) \otimes_{\dD'''} \hilb''' \arrow[d, "a"]            \\
(\hilb \otimes_{\dD'} \hilb') \otimes_{\dD''} (\hilb'' \otimes_{\dD'''} \hilb''') \arrow[rd, "a"', shorten >=50pt, pos=0.2]                          &                                                                                  & \hilb \otimes_{\dD'} ((\hilb' \otimes_{\dD''} \hilb'') \otimes_{\dD'''} \hilb''') \arrow[ld, "\id \otimes_{\dD'} a", shorten >=50pt, pos=0.2] \\
& |[label={[overlay]\hilb \otimes_{\dD'} (\hilb' \otimes_{\dD''} (\hilb'' \otimes_{\dD'''} \hilb'''))}]| &                                                                                                           
\end{tikzcd}
\]

\[
\begin{tikzcd}[row sep=1.5cm, column sep=0.7cm]
(\hilb \otimes_{\dD'} \dD') \otimes_{\dD'} \hilb' \arrow[rd, "r \otimes_{\dD'} \id"'] \arrow[rr, "a"] &                             & \hilb \otimes_{\dD'} (\dD' \otimes_{\dD'} \hilb') \arrow[ld, "\id \otimes_{\dD'} l"] \\
& \hilb \otimes_{\dD'} \hilb' &                                                                                   
\end{tikzcd}
\]
for any $(\hilb,\hilb',\hilb'',\hilb''')$ from $\homcat(\dD,\dD') \times \homcat(\dD',\dD'') \times \homcat(\dD'',\dD''') \times \homcat(\dD''',\dD'''')$, any objects $\dD$, $\dD'$, $\dD''$, $\dD'''$ and $\dD''''$ from $\hombicat(\vnA, \vnB)$ and any commutative algebras $\vnA$ and $\vnB$.

\end{proof}

\subsubsection{The tricategory} \label{subsec:tric}

Now, we have all ingredients necessary to make a definition of a tricategory.

\begin{definition} \label{def:comphombicat}
    Let $\vnA$, $\vnB$ and $\vnC$ be commutative algebras.
    We define a functor $\otimes_\vnB$ from $\hombicat(\vnB,\vnC) \times \hombicat(\vnA,\vnB)$ to $\hombicat(\vnA,\vnC)$ as follows:
    
    \begin{enumerate}
        \item The functor associates to any object $(\dE,\beta,\dD,\alpha) \in \hombicat(\vnB,\vnC) \times \hombicat(\vnA,\vnB)$ the object $(\dD \otimes_\vnB \dE, \alpha * \beta) \in \hombicat(\vnA,\vnC)$ where $\alpha * \beta$ maps any $(a,c) \in \vnA \otimes \vnC$ to the element $\alpha(a,1) \otimes_\vnB \beta(1,c) \in Z(\dD \otimes_\vnB \dE)$.

        \item For any objects $(\dE,\beta,\dD,\alpha)$ and $(\dG,\delta,\dF,\gamma)$ from $\hombicat(\vnB,\vnC) \times \hombicat(\vnA,\vnB)$, the functor $\otimes_\vnB$ associates the functor $F(\dG,\dE,\dF,\dD)$ from $\homcat(\dG,\dE) \times \homcat(\dF,\dD)$ to $\homcat(\dD \otimes_\vnB \dE, \dF \otimes_\vnB \dG)$ which maps any object $(\hilbK,\hilb)$ to $\hilb \otimes_\vnB \hilbK$ (which is well-defined since there is a unique canonical $\vnB$-module structure on $\hilb$ and on $\hilbK$) and any morphism $(\psi\colon \hilbK\to\hilbK', \phi\colon \hilb\to\hilb')$ to $\phi \otimes_\vnB \psi$.

        \item For each $(\dE,\beta,\dD,\alpha) \in \hombicat(\vnB,\vnC) \times \hombicat(\vnA,\vnB)$, the functor $\otimes_\vnB$ associates the identity of $\dD \otimes_\vnB \dE$.

        \item \label{def:it:exchlaw}
        For any objects $(\dE,\beta,\dD,\alpha)$, $(\dE',\beta',\dD',\alpha')$ and $(\dE'',\beta'',\dD'',\alpha'')$ from $\hombicat(\vnB,\vnC) \times \hombicat(\vnA,\vnB)$, the functor $\otimes_\vnB$ associates a natural transformation from the functor $\otimes_{\dD' \otimes_\vnB \dE'} \circ  F(\dE'',\dE',\dD'',\dD') \times F(\dE',\dE,\dD',\dD)$ to the functor $F(\dE'',\dE,\dD'',\dD) \circ \otimes_{\dE'} \times \otimes_{\dD'}$ that maps any object $(\hilbK',\hilb',\hilbK,\hilb)$ to the following isomorphism:
        \[
        \Phi\colon
        \begin{array}{ccc}
            (\hilb \otimes_\vnB \hilbK) \otimes_{\dD' \otimes_\vnB \dE'} (\hilb' \otimes_\vnB \hilbK') & \rightarrow & (\hilb \otimes_{\dD'} \hilb') \otimes_\vnB (\hilbK \otimes_{\dE'} \hilbK')\\
            (h \otimes_\vnB k) \otimes_{\dD' \otimes_\vnB \dE'} (h' \otimes_\vnB k') & \mapsto &  (h \otimes_{\dD'} h') \otimes_\vnB (k \otimes_{\dE'} k')
        \end{array}.
        \]
    \end{enumerate}
\end{definition}

\begin{proof}[Proof of the well-definition]

One can easily check that the axioms from \cite{Bicategory} are satisfied by the functor of bicategories defined in \Cref{def:comphombicat}.
Indeed, for any $\hilb$ from $\homcat(\dD,\dD')$, $\hilb'$ from $\homcat(\dD',\dD'')$, $\hilb''$ from $\homcat(\dD'',\dD''')$, $\hilbK$ from $\homcat(\dE,\dE')$, $\hilbK'$ from $\homcat(\dE',\dE'')$, $\hilbK''$ from $\homcat(\dE'',\dE''')$, any $\dD$, $\dD'$ and $\dD''$ from $\hombicat(\vnA,\vnB)$, any $\dE$, $\dE'$ and $\dE''$ from $\hombicat(\vnB,\vnC)$ and any commutative algebras $\vnA$, $\vnB$ and $\vnC$, the following diagram commutes:
\[
\begin{tikzcd}[row sep=2cm]
\hilb \otimes_\vnB \hilbK                                                                                                                & (\hilb\otimes_{\dD'}\dD') \otimes_\vnB (\hilbK\otimes_{\dE'}\dE') \arrow[l, "r \otimes_\vnB r"]          \\
(\hilb \otimes_\vnB \hilbK) \otimes_{\dD' \otimes_\vnB \dE'} (\dD' \otimes_\vnB \dE') \arrow[r, "\id \otimes_\vnB \id"'] \arrow[u, "r"] & (\hilb \otimes_\vnB \hilbK) \otimes_{\dD' \otimes_\vnB \dE'} (\dD' \otimes_\vnB \dE') \arrow[u, "\Phi"']
\end{tikzcd}
\]
and so do its left analogous and the following diagram:
\begin{center}
\begin{sideways}
\begin{tikzcd}[row sep=2cm, column sep=0.26cm]
(\hilb \otimes_\vnB \hilbK) \otimes_{\dD' \otimes_\vnB \dE'}
((\hilb' \otimes_\vnB \hilbK')
\otimes_{\dD'' \otimes_\vnB \dE''}
(\hilb'' \otimes_\vnB \hilbK''))
\arrow[d, "\id \otimes_\vnB \Phi"']
&
((\hilb \otimes_\vnB \hilbK)
\otimes_{\dD' \otimes_\vnB \dE'}
(\hilb' \otimes_\vnB \hilbK'))
\otimes_{\dD'' \otimes_\vnB \dE''}
(\hilb'' \otimes_\vnB \hilbK'')
\arrow[d, "\Phi \otimes_\vnB \id"] \arrow[l, "a"', pos=0.4]
\\
(\hilb \otimes_\vnB \hilbK) \otimes_{\dD' \otimes_\vnB \dE'}
((\hilb' \otimes_{\dD''} \hilb'')
\otimes_\vnB
(\hilbK' \otimes_{\dE''} \hilbK''))
\arrow[d, "\Phi"']
&
((\hilb \otimes_{\dD'} \hilb')
\otimes_\vnB
(\hilbK \otimes_{\dE'} \hilbK'))
\otimes_{\dD'' \otimes_\vnB \dE''}
(\hilb'' \otimes_\vnB \hilbK'')
\arrow[d, "\Phi"]                                 
\\
((\hilb \otimes_{\dD'} (\hilb' \otimes_{\dD''} \hilb''))
\otimes_\vnB
(\hilbK \otimes_{\dE'} (\hilbK' \otimes_{\dE''} \hilbK'')))
&
((\hilb \otimes_{\dD'} \hilb') \otimes_{\dD''} \hilb'')
\otimes_\vnB
((\hilbK \otimes_{\dE'} \hilbK') \otimes_{\dE''} \hilbK'')
\arrow[l, "a \otimes_\vnB a"']                                  
\end{tikzcd}
\end{sideways}
\end{center}

\end{proof}

\begin{definition}[The tricategory $\comalg$] \label{def:comalg}

    We define the tricategory of commutative algebras $\comalg$ as follows:
    \begin{enumerate}
        \item Objects: An object is a commutative algebra.

        \item Hom-bicategory: The hom-bicategory associated to two objects $\vnA$ and $\vnB$ is $\hombicat(\vnA,\vnB)$ defined in \Cref{def:hombicat}.

        \item Composition: The composition functor between $\hombicat(\vnA,\vnB)$ and $\hombicat(\vnB,\vnC)$ is defined by \Cref{def:comphombicat} for any objects $\vnA$, $\vnB$ and $\vnC$.

        \item Identity $1$-morphisms: For each object $\vnA$ its identity is $(\vnA,\mu)$ where $\mu$ is the multiplication map of $\vnA$.

        \item Compositors: For any given three $1$-morphisms $(\dD,\alpha\colon \vnA \otimes \vnA' \to Z(\dD))$, $(\dE,\beta\colon \vnA' \otimes \vnA'' \to Z(\dE))$ and $(\dF,\gamma\colon \vnA'' \otimes \vnA''' \to Z(\dF))$, the compositors are $a_{\dD,\dE,\dF} = (\dD \otimes_{\vnA'} \dE) \otimes_{\vnA''} \dF$ and $a_{\dD,\dE,\dF}^{{\bbox}} = \dD \otimes_{\vnA'} (\dE \otimes_{\vnA''} \dF)$.

        \item Associator modifications: For any objects $\vnA$, $\vnA'$, $\vnA''$ and $\vnA'''$, for any $1$-morphisms $(\dD,\alpha)$ and $(\dD',\alpha')$ from $\vnA$ to $\vnA'$, $(\dE,\beta)$ and $(\dE',\beta')$ from $\vnA'$ to $\vnA''$ and $(\dF,\gamma)$ and $(\dF',\gamma')$ from $\vnA''$ to $\vnA'''$ and for any $\dD$--$\dD'$ bimodule $\hilb$, $\dE$--$\dE'$ bimodule $\hilb'$ and $\dF$--$\dF'$ bimodule, the interchanger associators
        \[
        a_{\hilb,\hilb',\hilb''} \colon \hilb_L \otimes_{a_{\dD',\dE',\dF'}} a_{\dD',\dE',\dF'} \simeq a_{\dD,\dE,\dF} \otimes_{a_{\dD,\dE,\dF}^{\bbox}} \hilb_R
        \]
        and
        \[
        a_{\hilb,\hilb',\hilb''}^{\bbox} \colon \hilb_R \otimes_{a_{\dD',\dE',\dF'}^{\bbox}} a_{\dD',\dE',\dF'}^{\bbox} \simeq a_{\dD,\dE,\dF}^{\bbox} \otimes_{a_{\dD,\dE,\dF}} \hilb_L
        \]
        are the obvious ones (see remark below), where $\hilb_L = (\hilb \otimes_{\vnA'} \hilb') \otimes_{\vnA''} \hilb''$ and $\hilb_R = \hilb \otimes_{\vnA'} (\hilb' \otimes_{\vnA''} \hilb'')$.
        
        \item Left and right unitors: For each $1$-morphism $(\dD,\alpha\colon\vnA\otimes\vnB\to Z(\dD))$, we have the equality of left and right unitors $l_\dD = l_\dD^{\bbox} = r_\dD = r_\dD^{\bbox} = \dD$.
        
        \item Unit modifications: For any $2$-morphism $\hilb$ between $1$-morphisms $(\dD,\alpha)$ and $(\dE,\beta)$, we have $l_\hilb$, $l_\hilb^{\bbox}$, $r_\hilb$ and $r_\hilb^{\bbox}$ are the obvious $3$-isomorphisms (see remark below).

        \item Pentagonator: For any objects $\vnA$, $\vnA'$, $\vnA''$, $\vnA'''$ and $\vnA''''$, any $1$-morphisms $(\dD,\alpha)$ from $\vnA$ to $\vnA'$, $(\dE,\beta)$ from $\vnA'$ to $\vnA''$, $(\dF,\gamma)$ from $\vnA''$ to $\vnA'''$ and $(\dG,\delta)$ from $\vnA'''$ to $\vnA''''$, the pentagonator $\pi_{\dD,\dE,\dF,\dG}$ is an obvious isomorphism (see remark below).

        \item Unitor modifications and interchanger: For any given two $1$-morphisms $(\dD,\alpha\colon \vnA \otimes \vnB \to Z(\dD))$ and $(\dE,\beta\colon \vnB \otimes \vnC \to Z(\dE))$, the left unitor modification $\lambda_{\dD,\dE}$, right unitor modification $\mu_{\dD,\dE}$ and the interchanger $\rho_{\dD,\dE}$ are obvious $3$-isomorphisms (see remark below).
    \end{enumerate}
\end{definition}

\begin{remark}
    For any algebras $\dD$ and any left $\dD$-bimodule $\hilb$, we can naturally define the following isomorphism:
    \[
    \Phi\colon
    \begin{array}{ccc}
        \dD \otimes_\dD \hilb & \longrightarrow & \hilb \\
        d \otimes_\dD h & \longmapsto & d \cdot h
    \end{array}.
    \]
    where $\Phi$ is well-defined since $d \otimes_\dD h = 1 \otimes_\dD d \cdot h$ for any $(d,h) \in \dD \times \hilb$.
    The same stands symmetrically for right modules.
    
    Similarly, for any $\dD$--$\dE$, $\dE$--$\dF$ and $\dF$--$\dG$ bimodule $\hilb$, $\hilb'$ and $\hilb''$, we have the following isomorphism:
    \[
    \begin{array}{ccc}
        \hilb \otimes_\dE (\hilb' \otimes_\dF \hilb'') & \to & (\hilb \otimes_\dE \hilb') \otimes_\dF \hilb'' \\
        h \otimes_\dE (h' \otimes_\dF h'') & \to & (h \otimes_\dE h') \otimes_\dF h''
    \end{array}.
    \]

    These isomorphisms are considered as obvious and we do not make them explicit.
\end{remark}

The verification of the well-definition of $\comalg$ is as straightforward as the verification of the previous categorical definitions,.
One can easily check thanks to the remark above that the axioms from \cite[Definition 3.1.2]{Tricategory} are satisfied by $\comalg$.

A direct comparison between the description of $\catconstCN$ obtained in \Cref{ssec:loccons} and the definition of $\comalg$ given in this section leads to our main result:

\begin{theorem}
    The tricategories $\catconstCN$ and $\comalg$ are isomorphic.
\end{theorem}

\section{Conclusion}

We introduced the tricategory $\catCN$ as an algebraic framework for studying the algebraic properties of circle-based conformal nets.
This setting was designed to be entirely concrete --- in order to avoid a mix between circle-based and coordinate-free definitions --- and still flexible thanks to the extensive use of categorical tools.
We proposed a way to implement solitons as defects in $\catCN$.
Furthermore, we provided a complete description of the locally constant subtricategory $\catconstCN$ of $\catCN$, identifying it with the tricategory $\comalg$.

\section*{Acknowledgement}

I am grateful to my supervisor Domenico Fiorenza for introducing me to the fascinating field of conformal nets, which allowed me to discover a new facet of active mathematical research.
I also thank Matthias Ludewig for the enriching and very helpful discussions we had.
Finally I would like to thank the Department of Mathematics "Guido Castelnuovo" at Sapienza University of Rome where I carried out this internship.

\printbibliography

\end{document}